\documentclass[11pt]{amsart}

\usepackage{hyperref}
\usepackage{amsfonts}
\usepackage{graphicx}
\usepackage{tabularx}
\usepackage{array}
\usepackage[usenames,dvipsnames]{color}
\usepackage{comment}
\usepackage{amsmath}
\usepackage{amsthm}
\usepackage{amssymb}
\usepackage{fullpage}
\usepackage[dvipsnames]{xcolor}
\usepackage{tikz}
\usepackage{listings}
\usepackage{dsfont}
\usepackage{enumerate}
\usepackage{romannum}

\newtheorem{theorem}{Theorem}[section]
\newtheorem{proposition}[theorem]{Proposition}

\newtheorem{conjecture}[theorem]{Conjecture}
\newtheorem{problem}[theorem]{Problem}

\newtheorem{lemma}[theorem]{Lemma}

\theoremstyle{definition}

\def\epsilon{\varepsilon}
\DeclareMathOperator{\ce}{f_{ce}}
\DeclareMathOperator{\Cce}{g_{ce}}
\DeclareMathOperator{\re}{f_{re}}

\title{Tight Bounds for Cycle-Edge Decompositions and Covers}
\author{Saieed Akbari}
\address{Department of Mathematical Sciences, Sharif University of Technology}
\email{s\_akbari@sharif.edu}
\author{Jonny Aloni}
\address{Department of Mathematics, Simon Fraser University}
\email{jonny\_aloni@sfu.ca}
\author{Arash Beikmohammadi}
\address{Department of Computer Science, Simon Fraser University}
\email{arash\_beikmohammadi@sfu.ca}
\author{Alexander Clow}
\address{Department of Mathematics, Simon Fraser University}
\email{alexander\_clow@sfu.ca}
\date{}

\begin{document}
\pagenumbering{arabic}

\begin{abstract}
    An old conjecture of Erd{\H{o}}s and Gallai states that  every $n$ vertex graph can be decomposed,
    that is $E(G)$ can be partitioned,
    into $O(n)$ cycles and edges.
    The covering version of this conjecture was proven by Pyber in 1985, where it was shown that all graphs can be covered by $n-1$ cycles and edges.
    The best upper bound on the number of cycles and edges required to decompose any graph is $O(n\log^*(n))$,
    which was recently shown by
    Buci{\'c} and Montgomery in 2023.
    Here $\log^*(n)$ denotes the iterated logarithm function.
    Meanwhile, a construction of Erd{\H{o}}s
    demonstrate that there exists graphs which require $(\frac{3}{2}-o(1))n$
    cycles and edges to be decomposed.
    We prove all graphs with maximum degree at most $4$ can be decomposed into
    $n-1$ or fewer cycles and edges.
    We also show that every $n$ vertex claw-free graph can be decomposed into $n-1$ or fewer $2$-regular subgraphs and edges.
    Finally, we prove that every graph $G$ containing a cycle can be covered by $n-2$ or fewer cycles and edges.
    This improves Pyber's covering theorem by proving that $n-1$ cycles and edges are required only for trees.
\end{abstract}

\maketitle

\section{Introduction}

Let $G = (V,E)$ be a simple graph.
If $\mathcal{G}$ is a set of graphs, then a \emph{decomposition} of $G$ into graphs from $\mathcal{G}$, is a partition $E_1,\dots, E_k$ of $E$
such that every subgraph $H_i = (V(E_i),E_i)$ is in $ \mathcal{G}$.
Here $V(E_i)$ denotes the subset of $V$ incident to edges in $E_i$.
Similarly, a \emph{cover} of $G$ by graphs from $\mathcal{G}$ is a set of subsets $E_1,\dots, E_k \subseteq E$ such that
the union of $E_1,\dots, E_k$ is $E$, and every subgraph $H_i = (V(E_i),E_i) \in \mathcal{G}$.
Hence, every decomposition of $G$ into graphs from $\mathcal{G}$ is a cover of $G$ by graphs from $\mathcal{G}$, but not visa-versa.
We will often identify a decomposition or cover $E_1,\dots, E_k$ of $G$ with the set of subgraphs $H_1,\dots, H_k$.
See Figure~\ref{fig:decomp and cover example} for an example of a decomposition and a cover.
For more definitions in graph theory we refer the reader to \cite{west2001introduction}.

\begin{figure}[h!]
    \centering
    \includegraphics[scale = 0.75]{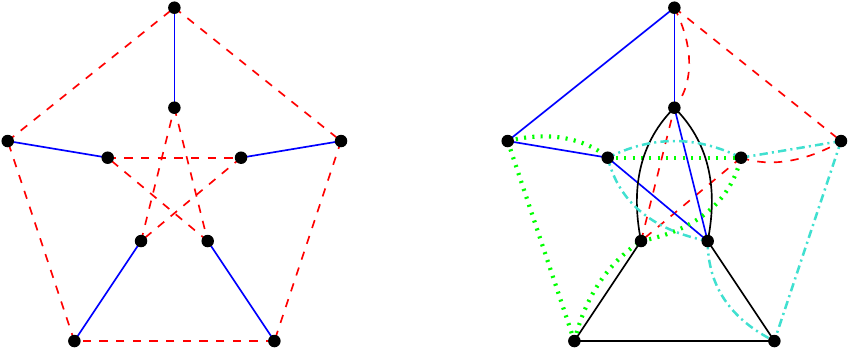}
    \caption{A decomposition of the Petersen graph into cycles and edges is given on the left.
    A cover of the Petersen graph by cycles is given on the right. If an edge is covered multiple times, then this is represented by a multiedge.}
    \label{fig:decomp and cover example}
\end{figure}

These notions of graph decomposition and cover are widely studied, see \cite{kottarathil2024graph}.
For many sets of graphs $\mathcal{G}$, not every graph can be decomposed into graphs from $\mathcal{G}$.
Moreover, the problem of determining if $G$ admits a decomposition into graphs from $\mathcal{G}$ is NP-hard \cite{dor1997graph}.
Our focus will not be determining if particular graphs $G$ have decompositions into graphs from some family $\mathcal{G}$.
Rather we will select our family $\mathcal{G}$, and at times our graphs $G$, such that we know a decomposition (or cover) exists, 
and our goal will be to prove $G$ can be decomposed (or covered) by a small number of graphs from our target family.

Such questions are well established in the literature.
Per \cite{lovasz1968covering}, Erd\H{o}s asked: what is the smallest number of paths required to decompose any $n$ vertex graph?
In response Gallai conjectured the following. If true, this bound is tight for any graph where every vertex has odd degree.

\begin{conjecture}[Gallai's Conjecture \cite{lovasz1968covering}]
    Every $n$ vertex connected graph $G$ can be decomposed into at most $\lceil \frac{n}{2} \rceil$ paths.
\end{conjecture}

This old conjecture remains open, but some related results and special cases are known.
For example Lov\'asz \cite{lovasz1968covering} proved that every $n$ vertex graph can be decomposed into $\lfloor \frac{n}{2} \rfloor$ cycles and paths.
More recently, Bonamy and Perrett \cite{bonamy2019gallai} proved Gallai's Conjecture for graphs with maximum degree at most $5$.
Gallai's Conjecture has also been verified when $G$ satisfies various structural conditions see \cite{botler2017path,fan2005path,favaron1988path,geng2015gallai,jimenez2017path}.

The primary focus on this paper will be a related conjecture of 
Erd\H{o}s and Gallai.
This conjecture appears in multiple of Erd\H{o}s' problem collections \cite{erdos1971some,erdos1973problems,erdHos1981combinatorial,erdHos1983some}
and has also been highlighted by other authors \cite{bondy1990small,bucic2023towards,conlon2014cycle,girao2021path,glock2016optimal,korandi2015decomposing,pyber1991covering,pyber1996covering}.

\begin{conjecture}[The Erd{\H{o}}s-Gallai Conjecture \cite{erdos1966representation}]\label{Conj: main}
    Every $n$ vertex graph $G$ can be decomposed into $O(n)$ cycles and edges.
\end{conjecture}

Constructions, first by Gallai \cite{erdos1966representation} and later by Erd\H{o}s \cite{erdHos1983some}
demonstrate that $(1+o(1))n$ cycles and edges are insufficient for general graphs.
In particular, Erd\H{o}s \cite{erdHos1983some} demonstrates a graph where $(\frac{3}{2}-o(1))n$ cycles and edges are required in a decomposition.

A simple argument, observed by Erd{\H{o}}s and Gallai in the 1960s, consisting of iteratively removing longest cycles 
gives an upper bound of $O(n\log(n))$.
The first improvement on this bound came in 2014 
from Conlon, Fox, and Sudakov \cite{conlon2014cycle}
who showed $O(n\log\log(n))$ cycles and edges are sufficient.
This was recently improved to $O(n\log^*(n))$ in 2023 by Buci{\'c} and Montgomery \cite{bucic2023towards}.
Here $\log^*(n)$ denotes the iterated logarithm of $n$.
The only cases where the Erd{\H{o}}s-Gallai conjecture has been proven 
is typical binomial random graphs \cite{conlon2014cycle,glock2016optimal, korandi2015decomposing},
and graphs with linear minimum degree \cite{conlon2014cycle,girao2021path}.

We are especially interested in graph classes where exactly $n-1$ or fewer cycles and edges are sufficient.
To this end, we define the following parameters.
For a graph $G$,
let $\ce(G)$ denote the minimum number of subgraphs that needed to decompose $G$ into cycles and edges, 
and let $\re(G)$ denote the minimum number of subgraphs that is needed to decompose $G$ into $2$-regular subgraphs and edges.
Similarly, let $g_{ce}(G)$ be the minimum number of subgraphs needed to cover $G$ by cycles and edges.

Given these notations, we can restate Conjecture~\ref{Conj: main} as asking if for all $n$ vertex graphs $G$, $\ce(G) = O(n)$.
In the same vein as  Bonamy and Perrett's result for Gallai's Conjecture on graphs with maximum degree $5$,
we prove that all graphs with maximum degree at most $4$ can be decomposed with $n-1$ cycles and edges.

\begin{theorem}\label{Thm: Max Degree 4}
    If $G$ is an $n$ vertex graph with maximum degree at most $4$, then 
    $\ce(G)\leq n-1$.
\end{theorem}

Next we consider the relaxed problem of decomposing into edges and graph families containing cycles.
We say a graph is \emph{even} if every vertex has even degree.
Then a graph is Eulerian if and only if it is connected and even.
We prove an analogue of Conjecture~\ref{Conj: main} for decompositions into even graphs and edges.

\begin{theorem}\label{Thm: Eulerian decomp}
    Every $n$ vertex graph containing a cycle can be decomposed into an even graph and at most $n-2$ edges.
\end{theorem}

Furthermore, we are able to show all such graphs can be decomposed into the smallest possible number of $2$-regular graphs.

\begin{theorem}\label{Thm: Even Delta}
    If $G$ is an even graph with maximum degree $\Delta$, then $G$ can be decomposed into $\frac{\Delta}{2}$, $2$-regular graphs.
    Hence, for every even graph $G$ with maximum degree $\Delta$, $\re(G) = \frac{\Delta}{2}$.
\end{theorem}

Next, we prove that all claw-free graphs can be decomposed into at most $n-1$, $2$-regular graphs and edges.
Notice this bound is still tight for all $n$, as demonstrated by paths.
Here if $G$ is $H$-free, then $G$ does not have $H$ as an induced subgraph.
Meanwhile, as is standard, the claw denotes the graph $K_{1,3}$.

\begin{theorem}\label{Thm: Claw-free}
    If $G$ is a claw-free $n$ vertex graph, then $\re(G)\leq n-1$.
\end{theorem}

The final set of relaxations of Conjecture~\ref{Conj: main} we consider is to study covers by cycles and edges, rather than decompositions into cycles and edges.
Recall that Pyber \cite{pyber1985erdHos} previous proved that $\Cce(G) \leq n-1$ for all graphs.
Note this upper bound is still best possible due to trees.
We are able to show the only graphs to require $n-1$ cycles or edges in a cover are trees.

\begin{theorem}\label{Thm: Cover-Cycles and Edges}
    If $G$ is an $n$ vertex graph that is not a tree, then $\Cce(G) \leq n-2$.
\end{theorem}

The rest of the paper is structured as follows.
In Section~\ref{Sec: Normal Problem} we will prove Theorem~\ref{Thm: Max Degree 4} which deals with decompositions into cycles and edges.
Next, Section~\ref{Sec: Eulerian and edges} deals with decompositions into even graphs and edges.
The next section, Section~\ref{Sec: 2-reg and edges}, covers the theorems related to decompositions into $2$-regular graphs and edges.
Our results on covers are given in Section~\ref{Sec: covers}.
We conclude with a discussion of future work.

\section{Cycle \& Edge Decompositions}
\label{Sec: Normal Problem}

We proceed directly to a proof of Theorem~\ref{Thm: Max Degree 4}.

\begin{proof}[Proof of Theorem~\ref{Thm: Max Degree 4}]
When $n$ is at most $3$ the result is trivial. 
Suppose then that $n$ is the least integer such that there exists an $n$ vertex counterexample to the theorem.
Fix an $n$ vertex graph $G$ with maximum degree at most $4$.

Let $C_1, \ldots, C_k$ be a longest sequence of edge-disjoint cycles in $G$ such that for any $i \in \{ 2,\ldots, k \}$
\begin{align*}
    V(C_i) \cap \bigcup_{j=1}^{i-1} V(C_j) \neq \emptyset.
\end{align*}
Among all such sequences of length $k$, choose one for which $|V(C_k)|$ is maximized.
Define a subgraph $H$ of $G$ with vertex set $V(H)=V(G)$ and edge set
$$E(H)=E(G) \setminus \bigcup_{i=1}^{k} E(C_i).$$
Suppose that for some $i \in \{2,\ldots,k\}$,
$$|V(C_i) \cap \bigcup_{j=1}^{i-1} V(C_j)| > 1.$$
Since $G$ has maximum degree $4$, such an $H$ has at least $k$ isolated vertices.
This is because every vertex that intersects two cycles in $C_i,C_j$ 
has $4$ incident edges in $\cup_{i=1}^k E(C_i)$.
Since $G$ has max degree $4$ each vertex is incident to at most $4$ edges,
implying each vertex incident intersecting two cycles $C_i,C_j$ will be isolated in $H$.

By the minimality of $n$, we can assume no component of $H$ with strictly less than $n$ vertices is a counterexample.
Hence, the edges of $H$ can be decomposed into at most $n-k-1$ cycles and edges. 
Including the $k$ cycles $C_1, \ldots, C_k$, we obtain a decomposition of $E(G)$ into at most
$k + (n - k - 1) = n - 1$
cycles and edges.

So we may assume that for all $i \in \{2,\ldots,k\}$,
$$|V(C_i) \cap \bigcup_{j=1}^{i-1} V(C_j)| = 1.$$

Let $H'$ be the subgraph of $H$ obtained by removing all isolated vertices. 
If $H'$ is disconnected, then each connected component can be handled separately. 
By the minimality of $n$, no connected component of $H'$ is a counterexample.
Since $H'$ has at least 2 components it can be decomposed into at most $|V(H')|-2 \leq n - k - 1$ cycles and edges. 
Including the $k$ original cycles, we again obtain a decomposition into at most $n - 1$ cycles and edges.

Otherwise, $H'$ is connected. 
Let $uv$ be an edge in $E(C_k)$ such that neither $u$ nor $v$ belongs to $\bigcup_{i=1}^{k-1} V(C_i)$.
Such vertices $u$ and $v$ exist since $C_k$ shares exactly $1$ vertex with the other cycles $C_1,\dots, C_{k-1}$, and $|V(C_k)|\geq 3$. 
If $u$ or $v$ are not in $H'$, then $H$ has at least $k$ isolated vertices, and as before, this leads to $\ce(G)\leq n-1$.
Suppose then that 
both $u$ and $v$ are vertices in $H'$.
Let
$$P_1: u = x_0, x_1, \ldots, x_a = v$$
be the path from $u$ to $v$ in $C_k$ that does not use the edge $uv$, and let
$$P_2: u = y_0, y_1, \ldots, y_b = v$$
be a path in $H'$ from $u$ to $v$, which exists since $H'$ is connected.

Since $P_2$ is a path in $H'$, the union of $P_1$ and $P_2$ forms a circuit in $G$. 
Let $W$ be this circuit.
If $W$ is a cycle, then 
$W$ is a longer cycle than $C_k$, since $V(P_1) = V(C_k)$, while $|V(P_2)|>2$, given $uv$ is not an edge in $H'$.
This contradicts the maximality of $|V(C_k)|$ we assumed at the beginning of the proof.

Otherwise, $W$ is not a cycle.
In this case $D = (V(P_1)\cup V(P_2),E(W))$ forms an Eulerian subgraph of $G$ with at least $2$ cycles.
Thus, $D$ can be decomposed into cycles $D_1,\dots, D_q,$ where $q\geq 2$, and
\[
V(D_1) \cap \bigcup_{j=1}^{k-1}V(C_i) \neq \emptyset,
\]
where for all $1 < i \leq q$, $V(D_i)\cap V(D_{i-1})\neq \emptyset$.
In this case, the sequence $C_1,\dots, C_{k-1}, D_1,\dots, D_q$ is a sequence of at least $k+1$ cycles such that every cycle has at least one vertex in common with previous cycles.
Since $q\geq 2$ this contradicts our assumption that $C_1,\dots, C_k$ was a longest sequences of cycles with this property.
This completes the proof.
\end{proof}

A related position that is of some interest is the following.

\begin{proposition}\label{Thm: 2k-reg}
    If $G$ is an $n$ vertex $2k$-regular graph with girth $g$, 
    then $\ce(G) \leq \frac{kn}{g}$.
\end{proposition}

\begin{proof}
    Let $G$ be a $2k$-regular, $n$ vertex, graph with girth $g$.
    By Petersen's theorem $G$ has a decomposition $D_1,\dots, D_k$ into $2$-factors.
    Let $i$ be fixed but arbitrary.
    Since $G$ has girth $g$ each component of $D_i$ has at least $g$ vertices.
    Hence, there are at most $\frac{n}{g}$ components in $D_i$, each of which is a cycle.
    Thus, $G$ has a decomposition into at most $\frac{kn}{g}$ cycles.
    This completes the proof.
\end{proof}

Notice it is an easy corollary of this argument
that if $G$ is a $6$-regular graph which can be decomposed into three $2$-factors, where at least one of these $2$-factors has at least one components with at least $4$ vertices, then $G$ can be decomposed into at most $n-1$ cycles and edge.
Is it possible every $6$-regular graph admits such a decomposition?

\section{Even \& Edge Decompositions}
\label{Sec: Eulerian and edges}

In this section we consider decompositions into even subgraphs and edges.
We begin by proving Theorem~\ref{Thm: Eulerian decomp}.

\begin{proof}[Proof of Theorem~\ref{Thm: Eulerian decomp}]
    Let $G$ be a graph with a cycle and suppose, for a contradiction, that we cannot decompose $G$ into an even subgraph and $n-2$ edges. 
    Now, consider an edge maximum even subgraph of $G$ called $H$.
    Let $Q$ be the graph $G-H$. We consider the cases where $Q$ contains a cycle and where $Q$ is a forest separately.
    
    Suppose $Q$ is a forest.
    If $Q$ is disconnected, then $Q$ has at most $n-2$ edges and the claim follows immediately.
    Otherwise $Q$ is a tree.
    Because $G$ contains a cycle, $H$ is non-empty and there exists an edge $uv\in E(H)$. 
    Have $P$ be the path from $u$ to $v$ in $Q$. Then $(H\cup P)- uv$ is an even subgraph that is larger then $H$, a contradiction.

    Suppose then that $Q$ contains a cycle.
    Let $C$ be this cycle.
    Then $H$ and $C$ are both edge-disjoint even subgraphs of $G$.
    The union of any edge disjoint even graphs is an even graph,
    so $H\cup C$ is an even subgraph of $G$.
    But this is a contradiction, since $H$ was edge maximum.
    This concludes the proof.
\end{proof}

We note that the result in Theorem~\ref{Thm: Eulerian decomp} is best possible for infinitely many 
graphs.
One such infinite family is given by gluing a tree onto each vertex of $K_4$.
In fact these are the only graphs with a cycle which cannot be decomposed into an even subgraph and $n-3$ or fewer edges.
We provide a proof of this now.

\begin{theorem}\label{Thm: n-3 even}
Let $G = (V,E)$ be a simple connected graph that contains a cycle. Then exactly one of the following are true
\begin{itemize}
\item[i)] $G$ is decomposable by an even subgraph and at most $n-3$ edges.
\item[ii)] $G$ is $K_4$ with a disjoint tree hanging off every vertex of the $K_4$.
\end{itemize}
\end{theorem}
\begin{proof}
By Theorem~\ref{Thm: Eulerian decomp} $G$ can be decomposed into an even graph and $n-2$ edges. 
If $G$ is a graph of type ii), then any even graph in such a decomposition will be of size at most 4 and thus an additional $n-2$ edges are needed in the decomposition.
So it remains to show that all graphs that contain a cycle which cannot be decomposed by $n-3$ edges are of type ii).

Let $G$ be an $n$ vertex graph that contains a cycle.
Let $H$ be an edge maximum even graph of $G$.
Then $F:=G - E(H)$ is a forest. 
By Theorem~\ref{Thm: Eulerian decomp} $F$ is not a tree, 
since this would imply $F$ has $n-1$ edges, 
thereby implying that $G$ cannot be decomposed into an even subgraph and $n-2$ edges, a contradiction.
If $F$ has at least $3$ connected components then it can be decomposed by $n-3$ edges. 
Suppose then that $F$ is the union of 2 disjoint trees $T_1$ and $T_2$. 

Note that for every edge $uv\in E(H)$ vertices $u$ and $v$ cannot be in the same component of $F$ denoted as $T_i$.
To see this, notice that if $u$ and $v$ are in $T_i$, then
there is a path $P$ in $ T_i$ that joins $u$ to $v$.
In this case $(H\cup P) - \{uv\}$ is a larger even graph than $H$.

So $H$ is a subgraph of the bipartite graph $(V,E(T_1,T_2))$,
where $E(T_1,T_2)$ is the set of edges with one endpoint in $T_1$ and the other in $T_2$.
Suppose that $H$ contains a matching of size $2$,
$u_1v_1,~u_2v_2\in E(H)$,
where $v_1,v_2\in V(T_1)$ and $u_1,u_2\in V(T_2)$.
Let $P_1$ be the path in $T_1$ that connects $v_1$ to $v_2$ and $P_2$ be the path that connects $u_1$ to $u_2$. 
The following graph is even $H^\prime:= (H\cup P_1\cup P_2) - \{u_1v_1,~u_2v_2\}$.
Then the maximality of $H$ implies $|E(H^\prime)| \leq |E(H)|$.
Hence, $v_1$ and $v_2$ are adjacent to each other in $T_1$, and $u_1$ and $u_2$ are adjacent in $T_2$. 

Suppose $H$ contains a matching of size $3$, call it $u_1v_1,u_2v_2,u_3v_3$.
Since $u_1,u_2,v_1,v_2$ in the last paragraph were chosen without loss of generality, we conclude that the maximality of $H$ implies $v_1,v_2,v_3$ and $u_1,u_2,u_3$ induce triangles in $F$.
But this is a contradiction since $F$ is a forest.
Hence, $H$ has no matching of size $3$.

Since $H$ is bipartite, K\H{o}nig's theorem implies the matching number of $H$, which is at most $2$, 
is equal to the vertex cover number of $H$.
Hence, the vertex cover number of $H$ is at most $2$.
Notice the vertex cover number
of $H$ must be greater than $1$, since $H$ contains a cycle by our assumption that $G$ contains a cycle.
Moreover, no connected component of $H$ is a star since $H$ is 2-regular. 
Therefore, $H$ has vertex cover number $2$, and $H$ has exactly one connected component containing $2$ or more vertices.
Suppose this component is $Q$.

Let $\{a,b\}\subseteq V$ be a vertex cover in $Q$.
We consider the cases 
$a \in V(T_1)$ and $b\in V(T_2)$, versus $a,b\in V(T_1)$ separately.
The labels for vertices and $T_1$ and $T_2$ are chosen without loss of generality.

\medskip
\noindent \underline{Case 1:} $a \in V(T_1)$ and $b\in V(T_2)$.
\medskip

Since the matching number of $Q$ is $2$, $\{b\}$ is not a vertex cover.
Thus there exists a vertex $u\in V(T_2)$ which is not $b$ and $ua\in E(H)$. The degree of $u$ is 1 because $u$ can only be adjacent to elements in the vertex cover $\{a,b\}$ and $u$ is not adjacent to $b$ because they are both in $T_2$.
This contradicts $H$ being even. So this case does not occur.

\medskip
\noindent \underline{Case 2:} $a,b \in V(T_1)$.
\medskip

Let $k = |V(T_2)\cap V(Q)|$.
Since $\{a,b\}$ is a vertex cover, all edges are incident to $a$ or $b$.
As $a,b \in V(T_1)$, the fact that $Q$ is bipartite implies $V(T_1)\cap V(Q) = \{a,b\}$.
Since $Q$ has matching number $2$
both vertices $a$ and $b$ have strictly positive degree.
Since $Q$ is even and bipartite, this implies $k\geq 2$.
If $k>2$ then there exists vertices $u_1,u_2,u_3$ in $T_2$.
Each of these vertices has a positive and even degree in $Q$.
Since $Q$ is bipartite, this implies $a$ and $b$ are both neighbours of each vertex $u_i$.

As with earlier in the proof, this implies $u_1,u_2,u_3$ forms a triangle in $G$.
Since $u_1,u_2,u_3$ are in $V(T_2)$, every edge of this triangle is in $T_2$, contradicting that $T_2$ is a tree.
So $k=2$.

Since $k=2$, $V(T_2)\cap Q = \{u_1,u_2\}$.
This forces $Q$ to have $4$ vertices.
Since $V(Q)$ induces a clique, $V(Q)$ induces $K_4$.
All other edges in $G$ are edges in $T$, a tree.
From here it is trivial to verify $G$ is of type ii).
\end{proof}

\section{$2$-Regular \& Edge Decompositions}
\label{Sec: 2-reg and edges}

In this section, we consider decompositions into $2$-regular graphs and edges.
Note that the $2$-regular subgraphs we discuss need not be spanning (i.e. they are not required to be $2$-factors).
We begin by considering decompositions of even graphs.
Then we consider claw-free graphs.

Recall that Theorem~\ref{Thm: Even Delta} is a generalization of Petersen's 2-factor theorem.
Our proof follows in the same spirit as most modern proofs of Peterson's theorem.

\begin{proof}[Proof of Theorem~\ref{Thm: Even Delta}]
Let $G$ be an Eulerian graph, with Eulerian circuit $W$. Construct a bipartite multi-graph $G_0$ with each part being copies of $V(G)$.
Vertices in opposite part associated to $v\in V(G)$ will be denoted as $v_1,v_2$.
An edge exists between $v_1$ and $u_2$ if, during our Eulerian circuit, we traverse $v$ and then $u$. 
Also add the multi-edge $v_1v_2$ $(\frac{\Delta}{2}-\deg{v})$ times, so the resulting graph is $\frac{\Delta}{2}$-regular.

Let $\mathcal{M} = \{M_1,\dots, M_{\frac{\Delta}{2}}\}$ be a decomposition of $G_0$ into perfect matchings.
Such a decomposition exists, since  $G_0$ is a $\frac{\Delta}{2}$-regular bipartite graph, see \cite{bondy1976graph} Corollary~5.2 page~73.

For each perfect matching $M\in\mathcal{M}$ we will create a $2$-regular subgraph $H$ of $G$ by the rule; 
$vu\in E(H)$ for $v\not=u$ if and only if
$v_1u_2\in M$, and if an edge $v_1v_2 \in M$ then $v\not \in V(H)$. Because all non $v_1v_2$ edges in $G_0$ are contained in some $H$, the set of $H$ graphs decomposes $G$.    
So $\re(G)\leq \frac{\Delta}{2}$.

Trivially, each vertex $v$ sees as at least $\frac{\deg(v)}{2}$ $2$-regular subgraphs and edges in a decomposition into $2$-regular subgraphs and edges.
Thus, $\re(G) \geq \frac{\Delta}{2}$.
\end{proof}

Now we consider claw-free graphs.
We begin with some helpful lemmas regarding the linkage of all vertices with odd degree.
The first of these deals with general graphs.
Using this we derive a stronger statement for claw-free graphs.

\begin{lemma}\label{Lemma: general odd paths}
    Let $G$ be a graph with $2k$ vertices of odd degree and let $P'_1,\dots, P'_k$ 
    be any set of paths that are not pairwise edge-disjoint, 
    such that for every vertex of odd degree $v$, there is a path $P'_i$ where $v$ is an endpoint of $P'_i$.
    Then there exists a set of edge-disjoint paths $P_1,\dots, P_k$
    such that for every vertex of odd degree $v$, 
    there is a path $P_i$ where $v$ is an endpoint of $P_i$,
    satisfying that $\sum_{i=1}^k |E(P_i)| <  \sum_{i=1}^k |E(P'_i)|$.
\end{lemma}

\begin{proof}
    Let $G$ be a graph with $2k$ vertices of odd degree.
    Let $\{P'_1,\dots, P'_k\}$ 
    be any set of paths that are not pairwise edge-disjoint, 
    such that for every vertex of odd degree $v$, there is a path $P'_i$ where $v$ is an endpoint of $P'_i$.
    Trivially such a set exists, since each connected component of $G$ contains an even number of odd degree vertices by the handshaking lemma.

    Let
    $\{P_1,\dots, P_k\}$ be such  a set of paths
    that minimizes the total number of edges in the union of all paths $P_i$.
    Then $\{P_1,\dots, P_k\}$ has  $|E(\cup_{i=1}^k P_i)| \leq  |E(\cup_{i=1}^k P'_i)| $.
    We claim that the minimality of $P_1,\dots, P_k$ implies this set of paths is edge-disjoint.
    If true, then $|E(\cup_{i=1}^k P_i)| <  |E(\cup_{i=1}^k P'_i)| $ since  $P'_1,\dots, P'_k$ is not edge-disjoint.

    Suppose two paths, say $P_1$ and $P_2$ without loss of generality, have a common edge $uv$.
    Since there are $2k$ vertices of odd degree, $k$ paths, and every vertex of odd degree is an endpoint of at least one path,
    every vertex of odd degree is the endpoint of exactly one path.
    Furthermore, all endpoints of paths $P_i$ are vertices of odd degree.
    Thus, $P_1$ and $P_2$ have distinct endpoints.
    
    Let $P_1$ and $P_2$ be given by
\begin{align*}
    &P_1: x_1,\dots, x_i = u,x_{i+1} = v, \dots, x_t \\
    &P_2: y_1,\dots, y_j = u,y_{j+1} = v, \dots, y_\ell.
\end{align*}
    Then $A: x_1,\dots, x_{i} = u = y_j, \dots, y_1$ and $B: x_t,\dots, x_{i+1} = v = y_{j+1}, \dots, y_\ell$ are paths with the same endpoints as $P_1$ and $P_2$.
    Furthermore, $|E(A)|+|E(B)| < |E(P_1)|+|E(P_2)|$.
    Thus, $\{A,B,P_3,\dots, P_k\}$ is a set of paths with the required property which spans fewer edges than $P_1,\dots, P_k$.
    This contradicts the minimality of $\{P_1,\dots, P_k\}$,
    completing the proof.
\end{proof}

\begin{lemma}\label{Lemma: Claw-free odd paths}
    Let $G$ be a claw-free graph with $2k$ vertices of odd degree and let $P'_1,\dots, P'_k$ 
    be any set of paths that are not pairwise vertex-disjoint, 
    such that for every vertex of odd degree $v$, there is a path $P'_i$ where $v$ is an endpoint of $P'_i$.
    Then there exists a set of vertex-disjoint paths $P_1,\dots, P_k$
    such that for every vertex of odd degree $v$, 
    there is a path $P_i$ where $v$ is an endpoint of $P_i$,
    satisfying that $|E(\cup_{i=1}^k P_i)| <  |E(\cup_{i=1}^k P'_i)| $.
\end{lemma}

\begin{proof}

    Let $G$ be a claw-free graph with  $2k$ vertices of odd degree.
    Let $\{P'_1,\dots, P'_k\}$ 
    be any set of paths that are not pairwise vertex-disjoint, 
    such that for every vertex of odd degree $v$, there is a path $P'_i$ where $v$ is an endpoint of $P'_i$.
    By Lemma~\ref{Lemma: general odd paths} all sets of paths which link the set of vertices of odd degree like this
    that minimize the number of edges used must be edge-disjoint paths.
    We now show that in a claw-free graph, such a minimum set of $k$ paths must be vertex-disjoint. 
    
    Among all such collections of $k$ edge-disjoint paths, choose one, call it $\{P_1,\dots, P_k\}$, that minimizes the total number of edges. 
    Suppose, for contradiction, that two of these paths, say $P_1$ and $P_2$, share a vertex $w$.

    Obviously,
    every vertex of odd degree is the endpoint of exactly one path.
    Thus, $w$ is not an endpoint of $P_1$ and $P_2$.
    Notice that $w$ may be the endpoint of $P_1$ or (exclusive) $P_2$.
    Without loss of generality suppose $w$ is not the endpoint of $P_1$.
    Then $w$ is adjacent to some distinct vertices $ a, b \in V(P_1) $ and to some vertex $c \in V(P_2) $. 
    Since $P_1$ and $P_2$ are edge-disjoint, $\{a,b\}\cap \{c\} = \emptyset$.

    Since $G$ is claw-free, at least one edge must exist among the vertices $\{a, b, c\}$.
   If $ab \in E(G)$, then instead of passing through $w$, we can replace the subpath $ awb $ in $P_1$ with the edge $ab$, resulting in a shorter path.

    By symmetry, we can argue about the edges $ac$ and $bc$ interchangeably.
    Consider the edge $ac$.
    Let $P_1$ and $P_2$ be given by
\begin{align*}
    &P_1: x_1,\dots, x_{i-1} = a, x_i = w,x_{i+1} = b, \dots, x_t \\
    &P_2: y_1,\dots, y_{j-1} = c, y_j = w, \dots, y_\ell.
\end{align*}
    Then $A: x_1,\dots, x_{i-1} = a, y_{j-1} = c, \dots, y_1$ and $B: x_t,\dots, x_{i} = w = y_{j}, \dots, y_\ell$ are paths with the same endpoints as $P_1$ and $P_2$.
    However $|E(A)|+|E(B)|<|E(P_1)|+|E(P_2)|$.
    Thus, $\{A,B,P_3,\dots, P_k\}$ is a set of paths whose endpoints consist of the set of odd degree vertices, which contains strictly less edges than $P_1,\dots, P_k$.
    If $\{A,B,P_3,\dots, P_k\}$ is a set of edge-disjoint paths, then this contradicts the minimality of $\{P_1,\dots, P_k\}$.
    Otherwise, $\{A,B,P_3,\dots, P_k\}$ is a set of paths that is not edge-disjoint.
    In this case, the proof of Lemma~\ref{Lemma: general odd paths} implies there is a set of edge-disjoint paths whose endpoints consist of odd degree vertices, say $\{P''_1,\dots, P''_k\}$ which spans strictly less edges than $\{A,B,P_3,\dots, P_k\}$ edges.
    Again this contradicts the minimality of $P_1,\dots, P_k$.
    
Thus, in a claw-free graph $G$, any set of $k$ paths that link the $2k$ vertices of odd degree in $G$ which minimizes the number of edges used
must be vertex-disjoint.
Hence, our assumption $\{P'_1,\dots, P'_k\}$ is not vertex-disjoint implies this $\{P'_1,\dots, P'_k\}$ does not minimize the number of edges used.
This completes the proof.
\end{proof}

We are now prepared to prove Theorem~\ref{Thm: Claw-free}.

\begin{proof}[Proof of Theorem~\ref{Thm: Claw-free}]
Suppose $G$ is a claw-free graph with maximum degree $\Delta$.
By the handshaking lemma, $G$ has an even number of vertices with odd degree.
Suppose without loss of generality that $G$ 
has $2k$ vertices of odd degree. 

Consider a set of paths $\{P_1, \ldots, P_{k}\}$ which links the $2k$ vertices of odd degree in $G$ with as few edges as possible. 
By Lemma~\ref{Lemma: Claw-free odd paths} this set of paths is vertex-disjoint.
If there exists a minimum set of paths of this type which contains a vertex of maximum degree, suppose without loss of generality that 
$\{P_1, \ldots, P_{k}\}$ contains a vertex of maximum degree.

By our choice of the paths $\{P_1, \ldots, P_{k}\}$,
the graph $H = G - (\cup_{i=1}^k E(P_i))$ is even.
Now, consider two cases based on whether $k > \frac{\Delta}{2}$ or $k \leq \frac{\Delta}{2}$.

\medskip
\noindent \underline{Case 1:} $k > \frac{\Delta}{2}$. 
\medskip

The number of edges in $P_1\cup\dots\cup P_k$ is at most $n-k$.
Since $H$ is even, 
Theorem~\ref{Thm: Even Delta}
implies 
$\re(H) = \frac{\Delta(H)}{2} \leq \frac{\Delta}{2}$.
Therefore,
\[
\re(G) \leq n - k + \frac{\Delta}{2} < n - \frac{\Delta}{2} + \frac{\Delta}{2} = n,
\]
as desired.

\medskip
\noindent \underline{Case 2:} $k \leq \frac{\Delta}{2}$. 
\medskip

Let $u$ be a vertex of degree $\Delta$. 
If there exists a vertex of degree $\Delta$ in $\cup_{i=1}^k V(P_i)$, then let $u$ be this vertex.
Suppose for a contradiction one of the paths $P_i$ shares 3 vertices with the neighbourhood of $u$,
call these vertices $x,y,z$.

First, consider the case where $u$ is not in $\cup_{i=1}^k V(P_i)$.
By our choice of $P_1,\dots, P_k$ this implies that no minimum set of paths connecting the $2k$ vertices of odd degree
contains a vertex of degree $\Delta$.
Observe that one of the two following things must be true.
First, $xyz$ is a subpath of $P_i$. 
In this case we may replace $xyz$ in $P_i$ with the path $xuz$ to form $P'_i$.
But then $\{P_1,\dots, P'_i,\dots P_k\}$ is a minimum set of paths connecting the $2k$ vertices of odd degree which
contains a vertex of degree $\Delta$. The existence of such a set is a contradiction.
Second, $xyz$ is not a subpath of $P_i$, in this case we suppose without loss of generality $y$ is contained on the subpath of $P_i$ connecting $x$ and $z$.
Replace this subpath of $P_i$ with $xuz$ to form $P'_i$.
Then $\{P_1,\dots, P'_i,\dots P_k\}$ is a set of paths connecting the $2k$ vertices of odd degree which has fewer edges than $\{P_1,\dots, P_k\}$
a contradiction.
Thus, we suppose $u$ is in $\cup_{i=1}^k V(P_i)$.

If $u$ has even degree, then we can shorten (or preserve the length of) $P_i$ by replacing the subpath between the furthest such neighbours $x$ and $z$ in $P_i$ with the path $xuz$ to form a path $P'_i$. 
Since $u$ is a vertex of even degree, $u$ is not an endpoint of $P_i$.
Hence, the resulting set of paths has the same endpoints as $\{P_1,\dots, P_k\}$ and no more edges than $\{P_1,\dots, P_k\}$.
If $P_i$ is not the path containing $u$, then this new set of paths is not vertex disjoint.
Hence,
Lemma~\ref{Lemma: Claw-free odd paths} implies $\{P_1,\dots, P'_i,\dots, P_k\}$ is not minimal.
This contradicts the minimality of $\{P_1,\dots, P_k\}$, so we suppose $P_i$ is the path containing $u$.
In this case, 
let $x$ and $y$ be the neighbours of $u$ in $P_i$.
Then without loss of generality, there is subpath connecting $y$ and $z$ in $P_i$.
Replace this subpath with the edge $uz$
to define $P'_i$.
Then $P'_i$ has strictly less edges than $P_i$ contradicting the minimality of $\{P_1,\dots, P_k\}$.

If $u$ has odd degree and $P_i$ is a path not containing $u$, 
then we arrive at a contradiction by the same argument as when $u$ has even degree and $P_i$ is a path not containing $u$.
Otherwise $u$ has odd degree and $P_i$ is the path containing $u$.
Since $u$ has odd degree, $u$ is an endpoint of $P_i$.
In this case, let $x$ be the neighbour of $u$ on $P_i$, and let $z \neq a$ be any other neighbours of $u$ on $P_i$.
Replace the subpath connecting $x$ and $z$ in $P_i$ with the edge $uz$.
Then, $\{P_1,\dots, P'_i,\dots, P_k\}$ is a set of paths connecting the $2k$ vertices of odd degree with strictly fewer edges than $\{P_1,\dots,P_k\}$.
This contradicts the minimality of $\{P_1,\dots, P_k\}$.
Therefore, we conclude that $|V(P_i) \cap N(u)| \leq 2$ for all paths $P_i$.

Since each path  $P_i$ intersects $N(u)$ in at most two vertices, the number of vertices in all $k$ paths is at most $n - \Delta + 2k$. 
Therefore, the number of edges removed, when deleting $\cup_{i=1}^k E(P_i)$ to form $H$, is at most
$
n - \Delta + k.
$
As before, $H$ is even and Theorem~\ref{Thm: Claw-free} implies  $ \re(H) = \frac{\Delta(H)}{2} \leq \frac{\Delta}{2}$. Hence,
\[
\re(G) \leq n - \Delta + k + \frac{\Delta}{2} \leq n - \frac{\Delta}{2} + \frac{\Delta}{2} = n.
\]

It remains to show that equality cannot hold. 
For equality to hold in the above bound, we must have \( 2k = \Delta \).
Furthermore, we must have exactly $n-\Delta+k$ edges in $\cup_{i=1}^k E(P_i)$.
For there to be this many edges in $\cup_{i=1}^k E(P_i)$, there must be exactly $n - \Delta + 2k = n$ vertices in 
$\cup_{i=1}^k V(P_i)$.
Hence, every vertex in \( G \) must appear in some path \( P_i \).
In that case, the maximum degree of $H$ would be strictly less than $\Delta$, implying that $H$ can be decomposed using strictly fewer than $ \frac{\Delta}{2}$, 2-regular subgraphs. This contradicts the assumption of equality, and thus we conclude that
\[
\re(G) < n,
\]
which completes the proof.
\end{proof}

\section{Covers by Cycles and Edges}
\label{Sec: covers}

In this section we turn our attention to covers rather than decompositions.
Notice that even when talking about covers, any tree requires $n-1$ cycles and edges to be covered.
We begin by noting the following interesting, and helpful,
result by Fan \cite{fan2003covers}.

\begin{theorem}\label{fan}[Fan \cite{fan2003covers}]
Let $G$ be an even $n$ vertex graph. Then $G$ can be covered with $\lfloor \frac{n-1}{2}\rfloor$ cycles.
\end{theorem}

Using this result, we can prove Theorem~\ref{Thm: Cover-Cycles and Edges}.
Recall that this states that every $n$ vertex graph $G$ can be covered by $n-1$ cycles and edges,
with equality if and only if $G$ is a tree.

\begin{proof}[Proof of Theorem~\ref{Thm: Cover-Cycles and Edges}]
If $G$ does not have a cycle, then it is a forest with at least two connected components; so $\Cce(G) \leq n-2$. 
So we may assume that $G$ contains a cycle.
Suppose $n$ is the least integer such that there exists an $n$ vertex graph $G$ such that $G$ contains a cycle and $\Cce(G) > n-2$.
Observe that $G$ must have at least $5$ vertices, 
as one can easily verify the theorem holds for graphs with at most $4$ vertices.
Let $G$ be such a smallest $n$ vertex counterexample.
Trivially, $G$ must be connected.

First, we show that if $G$ has a cut vertex, then $G$ is not a smallest counterexample.
Let
$v$ be a cut vertex in $G$.
Then there exists non-empty subgraphs $G_1$ and $G_2$ of $G$,
where
$V(G_1) \cap V(G_2) = \{v\}$, $G = G_1\cup G_2$, and there is no edge $uw$ where $u\in V(G_1)\setminus \{v\}$ and $w\in V(G_2)\setminus \{v\}$.
Let $G_1,G_2$ be such graphs and let
$|V(G_i)|=n_i,$ for $i=1,2$. 
If one of $G_1$ and $G_2$, say $G_1$ without loss of generality, is not a tree,  then by induction hypothesis, we have, $$\Cce(G)=\Cce(G_1)+\Cce(G_2)\leq (n_1-2)+ (n_2-1)=n-2,$$
as desired. If both $G_1$ and $G_2$ are trees, then $G$ is a tree, which would contradict $G$ containing a cycle.
Hence, $G$ does not contain a cut vertex, implying $G$ is $2$-connected

Suppose first that $G$ is a graph that can be decomposed into an even subgraph $H$ and a graph $F$ with at most $n-3$ edges. 
Since $G$ is $2$-connected every pair of edges in $G$ is contained in a cycle. 
Thus, $E(F)$ can be covered with at most $\lceil \frac{n-3}{2} \rceil$ cycles and edges.
From here, Theorem \ref{fan} implies
\[
\Cce(G) \leq \Cce(H)+ \lceil \frac{n-3}{2} \rceil \leq \lfloor \frac{n-1}{2} \rfloor + \lceil \frac{n-3}{2} \rceil=n-2. 
\]

Otherwise, Theorem~\ref{Thm: n-3 even} implies $G$ is $K_4$ with trees glued to each vertex.
Notice that such a graph $G$ is $2$-connected if and only if $G$ is $K_4$.
It is easy to verify $K_4$ is not a counterexample.
This completes the proof.
\end{proof}

Observe that this bound is best possible.
Notice the bound is reached by any connected graph $G$ which is the union of $k$ vertex disjoint trees $T_1,\dots, T_k$, and a graph $H$,
such that $\Cce(H) = |V(H)| - 2$, while $|V(H) \cap V(T_i)| = 1$ for all $i \in \{1,\dots, k\}$.
Are  these the only graphs that reach the bound?

\section{Future Work}

Of course the main open problem for future work is Conjecture~\ref{Conj: main}.
Proving the full conjecture is likely difficult, so we make note of some interesting, likely easier to solve, open cases.

\begin{problem}
    Show that if $G$ has maximum degree $5$, or if $G$ is $6$-regular, or $8$-regular,
    then $G$ can be decomposed into at most $n-1$ cycles and edges.
\end{problem}

Observe that to solve the $6$-regular case it is sufficient to prove the following conjecture.
We note that if $6$-regular graphs decomposing into three $2$-factors is replaced with $8$-regular graphs decomposing into four $2$-factors,
in this conjecture, 
then the result follows by a theorem Abreu, Aldred, Funk, Jackson, Labbate, and Sheehan \cite{abreu2004graphs}
regarding the existence of non-isomorphic $2$-factors in graphs with minimum degree at least $8$. 

\begin{conjecture}
    Every $6$-regular graph is decomposable into three $2$-factors, where one of these $2$-factors has a component with at least $4$ vertices.
\end{conjecture}

An interesting weakened version of Conjecture~\ref{Conj: main} it to consider the same problem for decompositions into $2$-regular graphs and edges.

\begin{conjecture}
    For all $n$ vertex graphs $G$, $\re(G) = O(n)$.
\end{conjecture}

A nice special case is to solve this conjecture for regular graphs.
Notice that Petersen's $2$-factor theorem implies it is sufficient to prove the result for $(2k+1)$-regular graphs.
We note that the result is easy to show for $(2k+1)$-regular graphs with a perfect matching.

\begin{conjecture}
    For all $n$ vertex $(2k+1)$-regular graphs $G$, $\re(G) \leq n-1$.
\end{conjecture}

Next, we formally state a problem mentioned at the end of Section~\ref{Sec: covers}.

\begin{problem}
    Characterize the $n$ vertex graphs $G$ such that $G$ contains a cycle and $\Cce(G) = n-2$.
\end{problem}

Another interesting avenue
of questions comes from the methods used in the paper.
In particular the matter of the existence of linkages between vertices of odd degree,
see Lemma~\ref{Lemma: general odd paths} and Lemmas~\ref{Lemma: Claw-free odd paths}.
What graph properties, other than being claw-free, can ensure there is a vertex disjoint linkage between the set of odd degree vertices in a graph $G$?
A natural question along these lines is to consider for what graphs $H$ must all $H$-free graphs have such a linkage?

\begin{problem}
    Characterize the graphs $H$ such that if $G$ is an $H$-free graph with $2k$ vertices of odd degree, then there exists vertex disjoint paths $P_1,\dots, P_{k}$, where for all odd degree vertices $v$, $v$ is the endpoint of some path $P_i$.
\end{problem}

\section*{Acknowledgement}

The research visit of S. Akbari at Simon Fraser University was supported in part by the ERC Synergy grant 
(European Union, ERC, KARST, project number 101071836).
Clow is supported by the Natural Sciences and Engineering Research Council of Canada (NSERC) through PGS D-601066-2025.
We would also like to thank Bertille Granet and Matija Buci{\'c} for independently making us aware of Erd{\H{o}}s and Gallai's conjecture, as well as the literature regarding it,
following the released of the first version of this paper.

\bibliographystyle{abbrv}
\bibliography{bib}

\end{document}